\documentclass[12pt,reqno]{amsart} 
\usepackage{amssymb,amscd,url}

\begin{document}

\allowdisplaybreaks


\title[FOM-versus-FOD for Dynamical Systems]
      {A Uniform Field-of-Definition/Field-of-Moduli Bound for Dynamical Systems on
          $\boldsymbol{\mathbb{P}^N}$}
\date{\today}
\author[J.R. Doyle]{John R. Doyle}
\email{jdoyle@latech.edu}
\address{Mathematics \& Statistics Department, Louisiana Tech
  University, Ruston, LA 71272 USA}
\author[J.H. Silverman]{Joseph H. Silverman}
\email{jhs@math.brown.edu}
\address{Mathematics Department, Box 1917
  Brown University, Providence, RI 02912 USA.
  ORCID: https://orcid.org/0000-0003-3887-3248}

\thanks{Silverman's research supported by Simons Collaboration Grant \#241309}



\hyphenation{ca-non-i-cal semi-abel-ian}


\newtheorem{theorem}{Theorem}
\newtheorem{lemma}[theorem]{Lemma}
\newtheorem{sublemma}[theorem]{Sublemma}
\newtheorem{conjecture}[theorem]{Conjecture}
\newtheorem{proposition}[theorem]{Proposition}
\newtheorem{corollary}[theorem]{Corollary}
\newtheorem*{claim}{Claim}

\theoremstyle{definition}
\newtheorem*{definition}{Definition}
\newtheorem*{intuition}{Intuition}
\newtheorem{example}[theorem]{Example}
\newtheorem{remark}[theorem]{Remark}
\newtheorem{question}[theorem]{Question}

\theoremstyle{remark}
\newtheorem*{acknowledgement}{Acknowledgements}


\newenvironment{notation}[0]{%
  \begin{list}%
    {}%
    {\setlength{\itemindent}{0pt}
     \setlength{\labelwidth}{4\parindent}
     \setlength{\labelsep}{\parindent}
     \setlength{\leftmargin}{5\parindent}
     \setlength{\itemsep}{0pt}
     }%
   }%
  {\end{list}}

\newenvironment{parts}[0]{%
  \begin{list}{}%
    {\setlength{\itemindent}{0pt}
     \setlength{\labelwidth}{1.5\parindent}
     \setlength{\labelsep}{.5\parindent}
     \setlength{\leftmargin}{2\parindent}
     \setlength{\itemsep}{0pt}
     }%
   }%
  {\end{list}}
\newcommand{\Part}[1]{\item[\upshape#1]}

\def\Case#1#2{%
\paragraph{\textbf{\boldmath Case #1: #2.}}\hfil\break\ignorespaces}

\renewcommand{\a}{\alpha}
\renewcommand{\b}{\beta}
\newcommand{\g}{\gamma}
\renewcommand{\d}{\delta}
\newcommand{\e}{\epsilon}
\newcommand{\f}{\varphi}
\newcommand{\fhat}{\hat\varphi}
\newcommand{\bfphi}{{\boldsymbol{\f}}}
\renewcommand{\l}{\lambda}
\renewcommand{\k}{\kappa}
\newcommand{\lhat}{\hat\lambda}
\newcommand{\m}{\mu}
\newcommand{\bfmu}{{\boldsymbol{\mu}}}
\renewcommand{\o}{\omega}
\renewcommand{\r}{\rho}
\newcommand{\rbar}{{\bar\rho}}
\newcommand{\s}{\sigma}
\newcommand{\sbar}{{\bar\sigma}}
\renewcommand{\t}{\tau}
\newcommand{\z}{\zeta}

\newcommand{\D}{\Delta}
\newcommand{\G}{\Gamma}
\newcommand{\F}{\Phi}
\renewcommand{\L}{\Lambda}

\newcommand{\ga}{{\mathfrak{a}}}
\newcommand{\gb}{{\mathfrak{b}}}
\newcommand{\gn}{{\mathfrak{n}}}
\newcommand{\gp}{{\mathfrak{p}}}
\newcommand{\gP}{{\mathfrak{P}}}
\newcommand{\gq}{{\mathfrak{q}}}

\newcommand{\Abar}{{\bar A}}
\newcommand{\Ebar}{{\bar E}}
\newcommand{\kbar}{{\bar k}}
\newcommand{\Kbar}{{\bar K}}
\newcommand{\Pbar}{{\bar P}}
\newcommand{\Sbar}{{\bar S}}
\newcommand{\Tbar}{{\bar T}}
\newcommand{\gbar}{{\bar\gamma}}
\newcommand{\lbar}{{\bar\lambda}}
\newcommand{\ybar}{{\bar y}}
\newcommand{\phibar}{{\bar\f}}

\newcommand{\Acal}{{\mathcal A}}
\newcommand{\Bcal}{{\mathcal B}}
\newcommand{\Ccal}{{\mathcal C}}
\newcommand{\Dcal}{{\mathcal D}}
\newcommand{\Ecal}{{\mathcal E}}
\newcommand{\Fcal}{{\mathcal F}}
\newcommand{\Gcal}{{\mathcal G}}
\newcommand{\Hcal}{{\mathcal H}}
\newcommand{\Ical}{{\mathcal I}}
\newcommand{\Jcal}{{\mathcal J}}
\newcommand{\Kcal}{{\mathcal K}}
\newcommand{\Lcal}{{\mathcal L}}
\newcommand{\Mcal}{{\mathcal M}}
\newcommand{\Ncal}{{\mathcal N}}
\newcommand{\Ocal}{{\mathcal O}}
\newcommand{\Pcal}{{\mathcal P}}
\newcommand{\Qcal}{{\mathcal Q}}
\newcommand{\Rcal}{{\mathcal R}}
\newcommand{\Scal}{{\mathcal S}}
\newcommand{\Tcal}{{\mathcal T}}
\newcommand{\Ucal}{{\mathcal U}}
\newcommand{\Vcal}{{\mathcal V}}
\newcommand{\Wcal}{{\mathcal W}}
\newcommand{\Xcal}{{\mathcal X}}
\newcommand{\Ycal}{{\mathcal Y}}
\newcommand{\Zcal}{{\mathcal Z}}

\renewcommand{\AA}{\mathbb{A}}
\newcommand{\BB}{\mathbb{B}}
\newcommand{\CC}{\mathbb{C}}
\newcommand{\FF}{\mathbb{F}}
\newcommand{\GG}{\mathbb{G}}
\newcommand{\NN}{\mathbb{N}}
\newcommand{\PP}{\mathbb{P}}
\newcommand{\QQ}{\mathbb{Q}}
\newcommand{\RR}{\mathbb{R}}
\newcommand{\ZZ}{\mathbb{Z}}

\newcommand{\bfa}{{\boldsymbol a}}
\newcommand{\bfb}{{\boldsymbol b}}
\newcommand{\bfc}{{\boldsymbol c}}
\newcommand{\bfd}{{\boldsymbol d}}
\newcommand{\bfe}{{\boldsymbol e}}
\newcommand{\bff}{{\boldsymbol f}}
\newcommand{\bfg}{{\boldsymbol g}}
\newcommand{\bfi}{{\boldsymbol i}}
\newcommand{\bfj}{{\boldsymbol j}}
\newcommand{\bfk}{{\boldsymbol k}}
\newcommand{\bfm}{{\boldsymbol m}}
\newcommand{\bfp}{{\boldsymbol p}}
\newcommand{\bfr}{{\boldsymbol r}}
\newcommand{\bfs}{{\boldsymbol s}}
\newcommand{\bft}{{\boldsymbol t}}
\newcommand{\bfu}{{\boldsymbol u}}
\newcommand{\bfv}{{\boldsymbol v}}
\newcommand{\bfw}{{\boldsymbol w}}
\newcommand{\bfx}{{\boldsymbol x}}
\newcommand{\bfy}{{\boldsymbol y}}
\newcommand{\bfz}{{\boldsymbol z}}
\newcommand{\bfA}{{\boldsymbol A}}
\newcommand{\bfF}{{\boldsymbol F}}
\newcommand{\bfB}{{\boldsymbol B}}
\newcommand{\bfD}{{\boldsymbol D}}
\newcommand{\bfG}{{\boldsymbol G}}
\newcommand{\bfI}{{\boldsymbol I}}
\newcommand{\bfM}{{\boldsymbol M}}
\newcommand{\bfP}{{\boldsymbol P}}
\newcommand{\bfX}{{\boldsymbol X}}
\newcommand{\bfY}{{\boldsymbol Y}}
\newcommand{\bfzero}{{\boldsymbol{0}}}
\newcommand{\bfone}{{\boldsymbol{1}}}

\newcommand{\Aut}{\operatorname{Aut}}
\newcommand{\Berk}{{\textup{Berk}}}
\newcommand{\Birat}{\operatorname{Birat}}
\newcommand{\Br}{\operatorname{Br}}
\newcommand{\codim}{\operatorname{codim}}
\newcommand{\Crit}{\operatorname{Crit}}
\newcommand{\critwt}{\operatorname{critwt}} 
\newcommand{\Cycle}{\operatorname{Cycles}}
\newcommand{\diag}{\operatorname{diag}}
\newcommand{\Disc}{\operatorname{Disc}}
\newcommand{\Div}{\operatorname{Div}}
\newcommand{\Dom}{\operatorname{Dom}}
\newcommand{\End}{\operatorname{End}}
\newcommand{\ExtOrbit}{\mathcal{EO}} 
\newcommand{\Fbar}{{\bar{F}}}
\newcommand{\Fix}{\operatorname{Fix}}
\newcommand{\FOD}{\operatorname{FOD}}
\newcommand{\FOM}{\operatorname{FOM}}
\newcommand{\Gal}{\operatorname{Gal}}
\newcommand{\GITQuot}{/\!/}
\newcommand{\GL}{\operatorname{GL}}
\newcommand{\GR}{\operatorname{\mathcal{G\!R}}}
\newcommand{\Hom}{\operatorname{Hom}}
\newcommand{\Index}{\operatorname{Index}}
\newcommand{\Image}{\operatorname{Image}}
\newcommand{\Isom}{\operatorname{Isom}}
\newcommand{\hhat}{{\hat h}}
\newcommand{\Ker}{{\operatorname{ker}}}
\newcommand{\Ksep}{K^{\text{sep}}}  
\newcommand{\Lift}{\operatorname{Lift}}
\newcommand{\limstar}{\lim\nolimits^*}
\newcommand{\limstarn}{\lim_{\hidewidth n\to\infty\hidewidth}{\!}^*{\,}}
\newcommand{\Mat}{\operatorname{Mat}}
\newcommand{\maxplus}{\operatornamewithlimits{\textup{max}^{\scriptscriptstyle+}}}
\newcommand{\MOD}[1]{~(\textup{mod}~#1)}
\newcommand{\Mor}{\operatorname{Mor}}
\newcommand{\Moduli}{\mathcal{M}}
\newcommand{\Norm}{{\operatorname{\mathsf{N}}}}
\newcommand{\notdivide}{\nmid}
\newcommand{\normalsubgroup}{\triangleleft}
\newcommand{\NS}{\operatorname{NS}}
\newcommand{\onto}{\twoheadrightarrow}
\newcommand{\ord}{\operatorname{ord}}
\newcommand{\Orbit}{\mathcal{O}}
\newcommand{\Pcase}[3]{\par\noindent\framebox{$\boldsymbol{\Pcal_{#1,#2}}$}\enspace\ignorespaces}
\newcommand{\Per}{\operatorname{Per}}
\newcommand{\Period}{\operatorname{Period}}
\newcommand{\Perp}{\operatorname{Perp}}
\newcommand{\PrePer}{\operatorname{PrePer}}
\newcommand{\PGL}{\operatorname{PGL}}
\newcommand{\Pic}{\operatorname{Pic}}
\newcommand{\Portrait}{\mathfrak{Port}}  
\newcommand{\Prob}{\operatorname{Prob}}
\newcommand{\Proj}{\operatorname{Proj}}
\newcommand{\Qbar}{{\bar{\QQ}}}
\newcommand{\rank}{\operatorname{rank}}
\newcommand{\Rat}{\operatorname{Rat}}
\newcommand{\Res}{\operatorname{Res}}
\newcommand{\Resultant}{\operatorname{Res}}
\newcommand{\Set}{\boldsymbol{P}} 
\renewcommand{\setminus}{\smallsetminus}
\newcommand{\sgn}{\operatorname{sgn}}
\newcommand{\shafdim}{\operatorname{ShafDim}}
\newcommand{\SL}{\operatorname{SL}}
\newcommand{\Span}{\operatorname{Span}}
\newcommand{\Spec}{\operatorname{Spec}}
\renewcommand{\ss}{{\textup{ss}}}
\newcommand{\stab}{{\textup{stab}}}
\newcommand{\Stab}{\operatorname{Stab}}
\newcommand{\Support}{\operatorname{Supp}}
\newcommand{\Sym}{\operatorname{Sym}}  
\newcommand{\TableLoopSpacing}{{\vrule height 15pt depth 10pt width 0pt}} 
\newcommand{\tors}{{\textup{tors}}}
\newcommand{\Trace}{\operatorname{Trace}}
\newcommand{\trdeg}{\operatorname{tr.deg.}}
\newcommand{\trianglebin}{\mathbin{\triangle}} 
\newcommand{\tr}{{\textup{tr}}} 
\newcommand{\UHP}{{\mathfrak{h}}}    
\newcommand{\val}{\operatorname{val}} 
\newcommand{\wt}{\operatorname{wt}} 
\newcommand{\<}{\langle}
\renewcommand{\>}{\rangle}

\newcommand{\pmodintext}[1]{~\textup{(mod}~#1\textup{)}}
\newcommand{\ds}{\displaystyle}
\newcommand{\longhookrightarrow}{\lhook\joinrel\longrightarrow}
\newcommand{\longonto}{\relbar\joinrel\twoheadrightarrow}
\newcommand{\SmallMatrix}[1]{%
  \left(\begin{smallmatrix} #1 \end{smallmatrix}\right)}


\begin{abstract}
Let $f:\mathbb{P}^N\to\mathbb{P}^N$ be an endomorphism of degree
$d\ge2$ defined over $\overline{\mathbb{Q}}$ or
$\overline{\mathbb{Q}}_p$, and let~$K$ be the field of moduli of~$f$.
We prove that there is a field of definition~$L$ for~$f$ whose
degree~$[L:K]$ is bounded solely in terms of~$N$ and~$d$. 
\end{abstract}

\maketitle

\tableofcontents

\section{Introduction}
\label{section:intro}
We start with an infomal description of a fundamental problem.
Let~$\Kbar$ be an algebraically closed field, for convenience of
characteristic~$0$, and let $X$ be an algebraic ``object'' defined
over~$\Kbar$.  The \emph{field of moduli} (FOM) of~$X$ is the smallest
subfield $K\subset\Kbar$ with the property that for every
$\s\in\Gal(\Kbar/K)$, there is a $\Kbar$-isomorphism from~$X^\s$
to~$X$. A \emph{field of definition} (FOD) for~$X$ is a subfield
$K\subset\Kbar$ with the property that there is an ``object'' $Y$
defined over~$K$ such that~$Y$ is $\Kbar$-isomorphic to~$X$. 
It is easy to see that every FOD contains the
FOM.  The \emph{field-of-moduli versus field-of-definition problem} is
to determine whether the FOM is itself already a FOD, and if not, to describe
the extent to which one must extend the FOM in order to obtain a FOD.

The FOM versus FOD problem arises in many areas of arithmetic
geometry, including for example the theories of abelian
varieties~\cite{MR0094360,MR0125113}, curves and their covering
maps~\cite{MR2181874,MR1443489}, sets of~$n$
points~\cite{MR3030517}, automorphic
functions on~$\PP^1$~\cite{shimura:fldofdef}, and dynamical
systems~\cite{silverman:fieldofdef}.  (This list of references is
meant to be illustrative, and is far from exhaustive.)  Our primary
goal in this paper is to prove a uniform bound for the minimal degree
of a FOD over the FOM for dynamical systems on~$\PP^N$.

We start with some notation and formal definitions, then we state our
main theorem and briefly survey earlier results on the FOM-versus-FOD
problem in dynamics.

\begin{notation}
\item[$K$]
  a field of characteristic $0$.
\item[$\Kbar$]
  an algebraic closure of $K$.  
\item[$G_K$]
  the Galois group $\Gal(\Kbar/K)$.
\item[$V/K$]
  an algebraic variety that is defined over $K$.
\item[$\End(V)$]
  the monoid of $\Kbar$-endomorphisms $f:V\to V$.
\item[$\Aut(V)$]
  the group of $\Kbar$-automorphisms $\f:V\to V$.
\end{notation}

We let $\Aut(V)$ act on $\End(V)$ by conjugation, i.e., for
$f\in\End(V)$ and $\f\in\Aut(V)$, we define
\[
  f^\f:=\f^{-1}\circ{f}\circ\f.
\]
This is the correct action for dynamics, since it commutes with iteration,
\[
  (f\circ f\circ\cdots\circ f)^\f = f^\f\circ f^\f\circ\cdots\circ f^\f.
\]

\begin{definition}
Let $f\in\End(V)$.  The \emph{field of moduli} (FOM) of~$f$ is the
fixed field of the following subgroup of~$G_K$:
\[
  \bigl\{ \s\in G_K : \text{there exists a $\f\in\Aut(V)$ so that $f^\s=f^\f$} \bigr\}.
\]
\end{definition}

\begin{definition}
Let $f\in\End(V)$.  A subfield $L$ of $\Kbar$ is a \emph{field of
  definition} (FOD) for~$f$ if there is an automorphism $\f\in\Aut(V)$
so that the conjugate~$f^\f$ is defined over~$L$.
\end{definition}

For a given $f\in\End(V)$, the following group of automorphisms of~$f$
plays a key role in studying the FOM and FODs for~$f$. More precisely,
the analysis is generally much easier to prove if one assumes
that~$\Aut(f)$ is trivial.

\begin{definition}
Let $f\in\End(V)$. The \emph{automorphism group of $f$} is
the subgroup of $\Aut(V)$ the commutes with~$f$, i.e.,
\[
  \Aut(f) := \bigl\{ \a\in\Aut(V) : f^\a = f \bigr\}.
\]
\end{definition}

It is clear that the~FOM of~$f$ is contained in every~FOD, but the~FOM
need not be a~FOD. The \emph{FOM-versus-FOD problem} is to
describe situations in which $\FOM=\FOD$, or to characterize the
amount by which they may differ.  The main result of the present note
is a uniform bound for the minimal degree of a FOD over the FOM for
endomorphisms of~$\PP^N$. Our bound applies to all maps, including
those having non-trivial automorphism group.  For ease of exposition,
we state a special case of our theorem here and refer the reader to
Theorem~\ref{theorem:FODFOMboundPNB} for the general statement.

\begin{theorem}
\label{theorem:FODFOMboundPN}
Fix integers $N\ge1$ and $d\ge2$. There is a constant $C(N,d)$ such
that the following holds\textup: Let $K$ be a number field or the
completion of a number field, and let $f:\PP^N\to\PP^N$ be an
endomorphism of degree~$d$ defined over~$\Kbar$ whose field of moduli
is contained in~$K$. Then there is a field of definition~$L$ for~$f$
satisfying
\[
  [L:K] \le C(N,d).
\]
\end{theorem}

For endomorphisms of~$\PP^1$, i.e., for $N=1$, much stronger results
are known. If we let~$C(N,d)$ denote the smallest value making
Theorem~\ref{theorem:FODFOMboundPN} true, then
\[
  C(1,d) = \begin{cases}
    1&\text{if $d$ is even~\cite{silverman:fieldofdef},}\\
    2&\text{if $d$ is odd~\cite{MR3230378}.}\\
  \end{cases}
\]
In other words, even degree self-maps of~$\PP^1$ have $\FOM=\FOD$,
while odd degree maps require at most a quadratic extension, and in
all odd degrees there do exist maps with $\FOM\ne\FOD$. In order to
handle maps having non-trivial automorphisms,
both~\cite{silverman:fieldofdef} and~\cite{MR3230378} require a
detailed case-by-case analysis using the classical classification of
finite subgroups of~$\PGL_2(\Kbar)$.

For maps $f:\PP^N\to\PP^N$ satisfying $\Aut(f)=1$, Hutz and
Manes~\cite{MR3309942} generalized the earlier $C(1,2d)=1$ result to
higher dimensions. It is also not hard in the setting of
Theorem~\ref{theorem:FODFOMboundPN} to show that if~$\Aut(f)=1$,
then~$f$ has a FOD of degree at most~$N+1$ over its FOM; see
Theorem~\ref{theorem:FODFOMboundPNB}(b).  But the situation becomes
significantly more complicated for maps~$f$ possessing non-trivial
automorphisms, and indeed Hutz and Manes give examples showing that
their main theorem is false for maps with $\Aut(f)\ne1$.

\begin{question}
\label{question:CNdjustN}
As noted earlier, Hidalgo~\cite{MR3230378} proved the $N=1$ case of
Theorem~\ref{theorem:FODFOMboundPN} with the explicit constant
$C(1,d)=2$. Thus our Theorem~\ref{theorem:FODFOMboundPN} may be viewed
as a higher dimensional version of Hidalgo's theorem, although our
result is neither as explicit nor as uniform as his~$\PP^1$ result,
and our general result (Theorem~\ref{theorem:FODFOMboundPNB}) further
requires a technical condition on the Brauer group of the base
field~$K$.  It is striking that Hidalgo's bound $C(1,d)=2$ does not
depend on~$d$. This raises the natural question of whether
Theorem~\ref{theorem:FODFOMboundPN} is true for all~$N$ with a
constant~$C(N,d)$ that depends only on~$N$.
\end{question}

\begin{remark}
A propos Question~\ref{question:CNdjustN}, we remark that 
Theorem~\ref{theorem:FODFOMboundPNB}(a) shows that the FOD/FOM bound
in Theorem~\ref{theorem:FODFOMboundPN} can be replaced with a bound of
the form
\begin{equation}
  \label{eqn:LKCNAutf}
  [L:K] \le C'\bigl(N,\#\Aut(f)\bigr).
\end{equation}
It is then a theorem of Levy~\cite{MR2741188} that $\#\Aut(f)$
may be bounded solely in terms of~$N$ and~$d$,
but~\eqref{eqn:LKCNAutf}  yields a stronger result if, for example,
one varies over a collection of maps of increasing degree whose
automorphism groups have bounded size.
\end{remark}

\begin{remark}
A primary application of the main result of this paper is to the
Uniform Boundeness
Conjecture~\cite{mortonsilverman:rationalperiodicpoints} for
preperiodic points. In a subsequent paper~\cite{moduliportrait2017} we
construct moduli spaces for dynamical systems with portraits, and we
use the FOD/FOM results from the present paper to relate the Uniform
Boundeness Conjecture to the existence of algebraic points of bounded
degree on these dynamical portrait moduli spaces.  We briefly describe
this connection in Section~\ref{section:moduli} and refer the reader
to~\cite{moduliportrait2017} for complete details.
\end{remark}
  

We close this introduction with a summary of the contents of this
paper and a brief sketch of the steps that go into the proof of
Theorem~\ref{theorem:FODFOMboundPN}.  As already noted,
Section~\ref{section:moduli} briefly discusses dynamical modulis
spaces the connection with the uniform boundedness conjecture. In
Section~\ref{section:prelimresultsgpcoho}, we review some facts about
Brauer groups and the period--index problem, and we prove a cohomology
splitting result (Proposition~\ref{proposition:ACG}) involving a
finite subgroup of an algebraic group and its normalizer and
centralizer. Section~\ref{section:FOMFODforVariety} deals with the
FOD/FOM problem for maps $f:V\to V$ of general varieties, and proves a
key criterion (Proposition~\ref{proposition:AfVffacts}) for the
$1$-cocycle $\f:G_K\to\Aut(V)$ associated to~$f$ to take values in the
normalizer of~$\Aut(f)$ in~$\Aut(V)$. In
Section~\ref{section:prelimresults} we state two Lemmas, which are
actually theorems of Brauer and Levy, that will be needed to prove our
main result. This leads to the proof in Section~\ref{section:fomfod}
of our main result, Theorem~\ref{theorem:FODFOMboundPNB}, which gives
a uniform FOD/FOM bound for all $f:\PP^N\to\PP^N$, and also a more
precise, and much more easily proven, FOD/FOM bound for maps
satisfying $\Aut(f)=1$. The proof of
Theorem~\ref{theorem:FODFOMboundPNB} involves successively moving the
$1$-cocycle from~$\PGL_{N+1}$ to the normalizer of~$\Aut(f)$
in~$\PGL_{N+1}$ to the centralizer of~$\Aut(f)$ in~$\PGL_{N+1}$. We
also lift~$\Aut(f)$ from~$\PGL_{N+1}$ to~$\GL_{N+1}$, decompose the
resulting representation into a sum of irreducible representations,
and apply a general verson of Schur's lemma and Hilbert's theorem~90
to map the~$1$-cocycle associated to~$f$ into a product of Brauer
groups.  Finally, in Section~\ref{section:FOMFODviaquotients} we prove
a result on endomorphisms, quotients, and twists
(Proposition~\ref{proposition:AfVffactsquotient}) and a result on
uniform existence of periodic points off of specified subvarieties
(Proposition~\ref{proposition:numbperptsonZ}) that we feel may be
useful in further study of dynamical FOD/FOM problems.

\section{Dynamical Moduli Spaces, FOM-versus-FOD, and the Dynamical Uniform Boundedness Conjecture}
\label{section:moduli}
This section indicates how the FOD/FOM bound in
Theorem~\ref{theorem:FODFOMboundPN} may be interpreted in terms of the
existence of algebraic points of bounded degree on fibers of dynamical
moduli spaces, and briefly descibes an application to the Uniform
Boundedness Conjecture.  We refer the reader
to~\cite{moduliportrait2017} for details of this application.  The
material in this section is not used elsewhere in this paper.

Let $\End_d^N$ denote the space of degree~$d$ endomorphisms
$f:\PP^N\to\PP^N$, and let~$\f\in\PGL_{N+1}(\Kbar)$ act
on~$\End_d^N(\Kbar)$ by conjugation.  The space~$\End_d^N$ has a
natural structure as an affine variety, and one can show that the
quotient $\Moduli_d^N:=\End_d^N\GITQuot\PGL_{N+1}$ also has the
structure of an affine variety in the sense of geometric invariant
theory. See~\cite{MR2741188,MR2567424,silverman:modulirationalmaps}
for details. We write $\langle\,\cdot\,\rangle:\End_d^N\to\Moduli_d^N$
for the quotient map. Then the FOM
of~$f\in\End_d^N(\Kbar)$ may equally well be defined as the field
generated by the coordinates of the
point~$\langle{f}\rangle\in\Moduli_d^N(\Kbar)$, and similarly a
field~$L$ is a FOD for~$f$ if~$\langle{f}\rangle$ is in
the image of $\End_d^N(L)$. The $\FOM\ne\FOD$ phenomenon arises due to
the fact that the map
\[
  \langle\,\cdot\,\rangle:\End_d^N(K)\to\Moduli_d^N(K)
\]
need not be surjective.

More generally, the authors have constructed spaces~$\End_d^N[\Pcal]$
and $\Moduli_d^N[\Pcal]$ that classify maps together with a list of
points modeling a given portrait~$\Pcal$;
see~\cite{moduliportrait2017}. These dynamical moduli spaces can be
used to formulate the following uniform boundedness conjecture.

\begin{conjecture}[Strong Moduli Boundedness Conjecture]
  \label{conjecture:strongMBC}
  Fix integers $D\ge1$, $N\ge1$, and $d\ge2$. Then there is a constant~$C_1(D,N,d)$
  such that for all number fields~$K/\QQ$ satisfying $[K:\QQ]\le D$ and
  all preperiodic portraits~$\Pcal$ containing at least~$C_1(D,N,d)$ points, we have
  \[
  \Moduli_d^N[\Pcal](K) = \emptyset.
  \]
\end{conjecture}

This may be compared with the usual uniform boundedness conjecture for
dynamical systems on~$\PP^N$.

\begin{conjecture}[Strong Uniform Boundedness Conjecture]
  \label{conjecture:strongUBC}
  \textup{(Silver\-man--Morton
    \cite{mortonsilverman:rationalperiodicpoints})} Fix integers $D\ge1$,
  $N\ge1$, and $d\ge2$. Then there is a constant~$C_2(D,N,d)$ such that for
  all number fields~$K/\QQ$ satisfying $[K:\QQ]\le D$ and all
  endomorphisms $f\in\End_d^N(K)$, we have
  \[
  \# \Bigl( \PrePer(f) \cap \PP^N(K) \Bigr) \le C_2(D,N,d).
  \]
  Here $\PrePer(f)$ denotes the set of points in $\PP^N(\Kbar)$ having
  finite forward $f$-orbit, i.e., the set of preperiodic points
  for~$f$.
\end{conjecture}

It is easy to see that Conjecture~\ref{conjecture:strongMBC} implies
Conjecture~\ref{conjecture:strongUBC}, but in order to prove the
converse, one needs a uniform FOD/FOM bound. And indeed, one of the
motivations for the present paper was to provide this key step in
proving the equivalence of Conjectures~\ref{conjecture:strongMBC}
and~\ref{conjecture:strongUBC} in~\cite{moduliportrait2017}.

\section{Preliminary Results on Group Cohomology and Brauer Groups}
\label{section:prelimresultsgpcoho}

We start with a standard result for finite Galois modules, whose
elementary proof we recall for the convenience of the reader.

\begin{lemma}
\label{lemma:H1fingp}
Let $A$ be a finite group with a continuous~$G_K$-action, and let
$c:G_K\to A$ be a continuous $1$-cocycle.  Then there exists an
extension $L/K$ satisfying
\[
  [L:K]\le \#A\cdot\#\Aut(A)\quad\text{and}\quad
  c_\s=1~\text{for all $\s\in G_L$.}
\]  
In particular, $[L:K]$ is bounded be a constant that depends only
on the order of the group~$A$.
\end{lemma}
\begin{proof}
The action of~$G_K$ on~$A$ is given by a group homomorphism
$G_K\to\Aut(A)$. The fixed field of the kernel of this homomorphism
has degree over~$K$ bounded by $\#\Aut(A)$.  Replacing~$K$ by this
fixed field, we may assume that~$G_K$ acts trivially on~$A$.  Then the
$1$-cocycle condition on~$c$ says that $c:G_K\to A$ is a
homomorphism. Taking~$L$ to be the fixed field of the kernel of this
homomorphism, we have $[L:K]\le\#A$, and the homomorphism~$c$ becomes
trivial on~$G_L$.
\end{proof}

We recall two definitions.

\begin{definition}
Let $\xi\in\Br(K)=H^2(G_K,\Kbar^*)$.  The \emph{period}, respectively
\emph{index}, of~$\xi$ are the quantities
\begin{align*} 
  \Period(\xi) &:= \text{the order of~$\xi$ as an element of~$\Br(K)$,} \\*
  \Index(\xi) &:= \min\bigl\{ [L:K] : \Res_{L/K}(\xi)=0~\text{in}~\Br(L) \bigr\}.
\end{align*}
\end{definition}

\begin{definition}
Let~$K$ be a field. We define the \emph{Brauer period-index exponent
  of~$K$} to be the smallest integer $\b(K)\ge1$ with the property
that that every element $\xi\in\Br(K)$ has the property that
\[
 \text{$\Index(\xi)$ divides  $\Period(\xi)^{\b(K)}$.}
\]
(If no such integer exists, we set $\b(K)=\infty$.)  We note that
the period always divides the index, so $\b(K)\ge1$, and thus
\[
  \text{$\Period(\xi)=\Index(\xi)$ for all $\xi\in\Br(K)$}
  \quad\Longleftrightarrow\quad \b(K)=1.
\]
See for example~\cite[Proposition~1.5.17]{GSM186}.
\end{definition}

\begin{remark}
\label{remark:perindsols}
We summarize some standard properties relating the period and the
index of elements of~$\Br(K)$. For additional information, see for
example~\cite{MR1233388}.
\begin{parts}
  \Part{(a)}
  If $K$ is a  global field or a local
  field,\footnote{Following~\cite{GSM186}, we define a \emph{local
      field} to be a finite extension of one of~$\RR$,~$\QQ_p$,
    or~$\FF_p(\!(t)\!)$, and a \emph{global field} to be a finite
    extension of~$\QQ$ or~$\FF_p(t)$.} then~$\b(K)=1$;
  see~\cite[Theorems~1.5.34~and~1.5.36]{GSM186}.
  \Part{(b)}
  Let $K$ be an extension of an algebraically closed field~$k$ of
  characteristic~$0$.  If $\trdeg(K/k)=1$, then Tsen's theorem says
  that $\Br(K)=0$, and if $\trdeg(K/k)=2$, then $\b(K)=1$;
  see~\cite{MR2060023}. More generally, it is known~\cite{MR2060023}
  that $\b(K)\ge\trdeg(K/k)-1$, and it is conjectured that this is
  always an equality.
\end{parts}
\end{remark}

\begin{proposition}
\label{proposition:ACG}
Let~$K$ be a field, and suppose that we are given the following
quantities\textup:
\begin{notation}
\item[$\Gcal/K$]
  an algebraic group defined over $K$.
\item[$\Acal/K$]
  a finite subgroup of $\Gcal(\Kbar)$ that is defined over $K$.
\item[$\Ncal/K$]
  the normalizer of $\Acal$ in $\Gcal(\Kbar)$.
\item[$\Ccal/K$]
  the centralizer of $\Acal$ in $\Gcal(\Kbar)$.
\item[$\xi$]
  a cohomology class in the pointed set $H^1(G_K,\Acal\backslash\Ncal)$.
\end{notation}
Then there is a finite extension $L/K$ and a constant $c=c(\#\Acal)$
depending only on the order of the group~$\Acal$ such that the following three statements are true\textup:
\begin{align}
  \Acal &\subset \Gcal(L). \\ 
  \Res_{L/K}(\xi) & \in\Image\Bigl(  H^1(G_L,\Ccal) \longrightarrow H^1(G_L,\Acal\backslash\Ncal) \Bigr). \\[1\jot]
  \label{eqn:LKlekapsupsup}
  [L:K] &\le  c\cdot \#(\Ccal\cap\Acal)^{\b(K)}.
\end{align}
\end{proposition}
\begin{proof}
To ease notation during the proof, when we replace~$K$ by an extension
field whose degree is bounded by a function of~$\#\Acal$, we again denote
the extension field by~$K$. We also let
\[
  m = \text{the exponent of the finite group~$\Acal$.}
\]

We first adjoin a primitive $m$'th root of unity to~$K$, which gives
an extension of degree at most~$\f(m)$, which is less than~$\#\Acal$.
Next, the fact that~$\Acal$ is finite and defined over~$K$ means that
the action of~$G_K$ on~$\Acal$ gives a homomorphism
$G_K\to\Aut(\Acal)$.  Hence replacing~$K$ with a finite extension
whose degree is bounded by~$\#\Aut(\Acal)$, we may assume that~$G_K$
acts trivially on~$\Acal$.  So we are reduced to the case
that~$\Acal\subset\Gcal(K)$ and~$\bfmu_m\subset K$.

For an abstract group~$G$ and subgroup~$A\subseteq G$ with normalizer~$N$
and centralizer~$C$, the elements of~$N$ induce (inner)
automorphisms of~$A$, so more-or-less by definition we have an exact
sequence
\begin{equation}
  \label{eqn:NtoInnAutA}
  \begin{CD}
    1 @>>> C @>>> N @>>\g\longmapsto (\a\mapsto\g^{-1}\a\g)> \Aut(A).
  \end{CD}
\end{equation}
We always have~$A\subset N$, but the inclusion $A\subset C$ is
equivalent to the statement that~$A$ is abelian.  So the exact
sequence~\eqref{eqn:NtoInnAutA}, taken modulo~$A$, yields
\begin{equation}
  \label{eqn:NtoInnAutAmodA}
  \begin{CD}
    1 @>>> A\backslash AC @>>> A\backslash N @>>\g\longmapsto (\a\mapsto\g^{-1}\a\g)> A\backslash\Aut(A).
  \end{CD}
\end{equation}
Applying~\eqref{eqn:NtoInnAutAmodA} with $G=\Gcal(\Kbar)$ and
$A=\Acal$, we find that
\begin{equation}
  \label{eqn:CbsNhookAutA}
  \Acal\Ccal\backslash \Ncal \longhookrightarrow \Acal\backslash\Aut(\Acal).
\end{equation}


We consider the exact sequence of groups
\[
\begin{CD}
  1
  @>>> \Acal\backslash\Acal\Ccal
  @>>> \Acal\backslash\Ncal
  @>>> \Acal\Ccal\backslash\Ncal
  @>>> 1.
\end{CD}
\]
Taking Galois cohomology gives the exact sequence of cohomology sets
\begin{equation}
\label{eqn:H1GKAACex}
\begin{CD}
  H^1(G_K,\Acal\backslash\Acal\Ccal)
  @>>> H^1(G_K,\Acal\backslash\Ncal)
  @>>> H^1(G_K,\Acal\Ccal\backslash\Ncal).
\end{CD}
\end{equation}

We know from~\eqref{eqn:CbsNhookAutA} that the group
$\Acal\Ccal\backslash\Ncal$ is finite and has order bounded
by~$\#\Acal\Aut(\Acal)$, so the order of $\Acal\Ccal\backslash\Ncal$
is bounded by a function of~$\#\Acal$.  Applying
Lemma~\ref{lemma:H1fingp}, we can replace~$K$ by a finite extension
such that the degree of the extension is bounded by a function
of~$\#\Acal$ and such that the image of~$\xi$ in
$H^1(G_K,\Acal\Ccal\backslash\Ncal)$ is trivial. Then the exact
sequence~\eqref{eqn:H1GKAACex} tells us that $\xi\in
H^1(G_K,\Acal\backslash\Acal\Ccal)$.

We use the basic isomorphism
\[
  \Acal\backslash\Acal\Ccal \cong (\Ccal\cap\Acal)\backslash\Ccal.
\]
The fact that $\Ccal\cap\Acal$ is in the center of~$\Ccal$ means that
when we take cohomology of the exact sequence
\[
\begin{CD}
  1 @>>> \Ccal\cap\Acal @>>> \Ccal @>>> (\Ccal\cap\Acal)\backslash\Ccal @>>> 1,
\end{CD}
\]
then as explained in~\cite[Chapter~VII, Appendix,
  Proposition~2]{MR554237}, we get an exact sequence with a connecting
homomorphism to an~$H^2$ term,
\begin{equation}
  \label{eqn:H1KCH2KCA}
\begin{CD}
  H^1(G_K,\Ccal) @>>> H^1(G_K,(\Ccal\cap\Acal)\backslash\Ccal) @>>> H^2(G_K,\Ccal\cap\Acal).
\end{CD}
\end{equation}
  
We  write the finite abelian group $\Ccal\cap\Acal$ as a product of
cyclic groups, say
\[
  \Ccal\cap\Acal \cong \bfmu_{n_1}\times\cdots\times\bfmu_{n_t}
\]
We note that this is an isomorphism of~$G_K$-modules, with all
$G_K$-actions trivial, since we have already arranged matters so
that~$G_K$ acts trivially on~$\Acal$ and on~$\bfmu_m$, and since
every~$n_i$ divides the exponent~$m$ of~$\Acal$. Hence the right-hand
cohomology group in the exact sequence~\eqref{eqn:H1KCH2KCA} is
\[
  H^2(G_K,\Ccal\cap\Acal) \cong  \prod_{i=1}^t H^2(G_K,\bfmu_{n_t}) \cong \prod_{i=1}^t \Br(K)[n_i].
\]
The image of~$\xi$ in $H^2(G_K,\Ccal\cap\Acal)$ gives a $t$-tuple
\[
  (\z_1,\ldots,\z_t) \in  \prod_{i=1}^t \Br(K)[n_i].
\]
The element~$\z_i$ has period~$n_i'$ for some integer dividing~$n_i$,
so by definition of the Brauer period-index exponent~$\b(K)$, we see
that~$\z_i$ becomes trivial over an extension of~$K$ of degree
dividing~$(n_i')^{\b(K)}$.  Applying this reasoning to each
of~$\z_1,\ldots,\z_t$ and taking the compositum of the fields, we see
that there is an extension~$L/K$ of degree at most
\[
  (n'_1n'_2\cdots n'_t)^{\b(K)}
  \le (n_1n_2\cdots n_t)^{\b(K)}
  = \#(\Ccal\cap\Acal)^{\b(K)}.
\]
such that the image of $\Res_{L/K}(\xi)$ in $H^2(G_L,\Ccal\cap\Acal)$
is trivial.

To recapitulate, we have constructed an extension~$L/K$
whose degree satisfies~\eqref{eqn:LKlekapsupsup}
and such that
\[
  \Res_{L/K}(\xi)\longrightarrow0\quad\text{in the cohomology group $H^2(G_L,\Ccal\cap\Acal)$.}
\]
It follows from the exact sequence~\eqref{eqn:H1KCH2KCA} the we can
lift~$\Res_{L/K}(\xi)$ to an element of the cohomology set
$H^1(G_L,\Ccal)$, which is the desired conclusion.
\end{proof}

\section{FOD/FOM for General Varieties}
\label{section:FOMFODforVariety}

We recall that we have fixed a field~$K$ of characteristic~$0$ and an
algebraic variety~$V/K$, and we are looking at morphisms $f:V\to V$
defined over an algebraic closure~$\Kbar$ of~$K$. To ease notation, we
let
\[
  \Acal_V := \Aut(V) \quad\text{and}\quad \Acal_f := \Aut(f),
\]
and we also define
\begin{align*}
  \Ncal_f &:= \text{the normalizer of $\Acal_f$ in $\Acal_V$}, \\*
  \Ccal_f &:= \text{the centralizer of $\Acal_f$ in $\Acal_V$}.
\end{align*}

Let $f:V\to V$ be an endomorphism whose field of moduli contains~$K$.
By definition of FOM, for each~$\s\in G_K$ there exists an
automorphism $\f_\s\in\Acal_V$ satisfying $f^\s=f^{\f_\s}$, and the
automorphism~$\f_\s$ is determined up to left composition by an
element of~$\Acal_f$. In this way~$f$ determines a well-defined
map\footnote{If we ever need to indicate the fact that~$\f$ depends
  on~$f$, we will write~$\f_{f,\s}$.}
\[
  \f : G_K\longrightarrow\Acal_f\backslash\Acal_V,\quad
  f^\s = f^{\f_\s}~\text{for all $\s\in G_K$.}
\]
From the definition, it is easy to verify
that~$\f$ is a ``$1$-cocycle relative to the subgroup~$\Acal_f$,'' i.e.,
it satisfies
\[
  \f_{\s\t}^{-1}\f_\t^{\vphantom1}\f_\s^\t\in\Acal_f\quad\text{for all $\s,\t\in G_K$.}
\]
In particular, if~$\Acal_f=1$, then~$\f$ is a $G_K$-to-$\Acal_V$
1-cocycle, and thus represents an element of the cohomology
set~$H^1(G_K,\Acal_V)$.  But in general~$\f$ is a sort of~$1$-cocycle
taking values in the quotient~$\Acal_f\backslash\Acal_V$, which need
not be a group.  However, if~$\Acal_f$ is defined over~$K$, then the
situation is better, which is the first part of the following
proposition.

\begin{proposition}
\label{proposition:AfVffacts}
With notation as above, we make the following two assumptions\textup:
\begin{align}
  \label{eqn:Affinite}
  &\text{\textbullet\enspace
    The automorphism group $\Acal_f$ is finite.} \\
  \label{eqn:AfdefK}
  &\text{\textbullet\enspace
    The group~$\Acal_f$ is defined over~$K$.} 
\end{align}
Then the following are true\textup:
\begin{parts}
\Part{(a)}
The image~$\f(G_K)$ of~$\f$ is contained in~$\Ncal_f$, the normalizer
of~$\Acal_f$ in~$\Acal_V$, and hence
\[
  \f : G_K \longrightarrow \Acal_f\backslash\Ncal_f
\]
is a $1$-cocycle taking values in a group. This in turn gives an
element of the cohomology set $H^1(G_K,\Acal_f\backslash\Ncal_f)$.
\Part{(b)}
The following are equivalent\textup:
\begin{itemize}
  \setlength{\itemsep}{0pt}
\item[(1)] There is a $\g\in\Acal_V$ such that $f^\g$ is defined over
  $K$, i.e.,~$K$ is a FOD for~$f$.
\item[(2)] There is a $\d\in\Acal_V$ such that
  $\f_\s=\Acal_f\d^{-1}\d^\s$ for all $\s\in G_K$, i.e.,~$\f$ is a
  $G_K$-to-$\Acal_f\backslash\Ncal_f$ coboundary.
\end{itemize}
\end{parts} 
\end{proposition}
\begin{proof}
(a)\enspace For~$\a\in\Acal_f\subset\Acal_V$ and~$\s\in G_K$, the
assumption~\eqref{eqn:AfdefK} says that $\a^\s\in\Acal_f$, which
allows us to compute
\[
  f^{\f_\s^{-1}\a\f_\s}
  = (f^{\s^{-1}})^{\a\f_\s}
  = \bigl( ( f^{\a^\s} )^{\s^{-1}} \bigr)^{\f_\s}
  = ( f^{\s^{-1}})^{\f_\s}
  = ( f^{\f_\s^{-1}})^{\f_\s}
  =f.
\]
Hence $\f_\s^{-1}\a\f_\s\in\Acal_f$, which proves that $\f_\s\in\Ncal_f$. Next,
for $\s,\t\in G_K$ we compute
\[
  f^{\f_{\s\t}} = f^{\s\t} = (f^{\f_\s})^\t = (f^\t)^{\f_\s^\t} = f^{\f_\t\f_\s^\t}.
\]
Hence
\[
  \f_{\s\t} \equiv \f_\t\f_\s^\t \pmod{\Acal_f},
\]
so $\f$ is a $G_K$-to-$\Acal_f\backslash\Ncal_f$ 1-cocycle.
\par\noindent(b)\enspace 
Suppose first that~(1) holds, so we have some $\g\in\Acal_V$ such that
$\f^\g$ is defined over $K$. It follows that for every $\s\in K$ we have
\[
  f^\g
  =  (f^\g)^\s
  = (f^\s)^{\g^\s}
  = (f^{\f_\s})^{\g^\s}
  = f^{\f_\s\g^\s}.
\]
Hence $\f_\s\g^\s\g^{-1}\in\Acal_f$, and we may take $\d=\g^{-1}$.

We next prove that~(2) implies~(1), so we assume that $\d\in\Acal_V$
has the property that $\f_\s=\Acal_f\d^{-1}\d^\s$ for all $\s\in G_K$.
We set $\g=\d^{-1}$, so $\f_\s\g^\s\g^{-1}\in\Acal_f$, and we use this
to compute
\[
  (f^\g)^\s
  = (f^\s)^{\g^\s}
  = (f^{\f_\s})^{\g^\s}
  = f^{\f_\s\g^\s}
  = f^\g.
\]
Hence~$f^\g$ is defined over~$K$.
\end{proof}

\section{Two Other Preliminary Results}
\label{section:prelimresults}

In this section we state two results that are needed for the proof of
Theorem~\ref{theorem:FODFOMboundPNB}. We denominate them as lemmas,
although they are in fact non-trivial theorems in their own right.

\begin{lemma}
\label{lemma:brauer}
\textup{(Brauer's Theorem)}
Let~$\Kbar$ be an algebraically closed field of characteristic~$0$,
let $\Gamma\subset\GL_{N+1}(\Kbar)$ be a finite group, let~$m$ be the
exponent of~$\Gamma$, and let~$\z_m\in\Kbar$ be a primitive $m$'th root
of unity. Then there exists an element $A\in\GL_{N+1}(\Kbar)$ such that
$A^{-1}\Gamma A\subset\GL_{N+1}\bigl(\QQ(\z_m)\bigr)$.
\end{lemma}
\begin{proof}
See, for example, \cite[Theorem~24, \S12.3]{MR0450380}.
\end{proof}

\begin{lemma}
\label{lemma:levy}
\textup{(Levy~\cite{MR2741188})} Let~$\Kbar$ be an algebraically
closed field of characteristic~$0$.  There is a constant $C_3(N,d)$
such that every $f\in\End(\PP^N)$ of degree~$d$ satisfies
\[
  \#\Aut(f) := \#\bigl\{\f\in\PGL_{N+1}(\Kbar):f^\f=f\bigr\}  \le C_3(N,d).
\]
\end{lemma}
\begin{proof}
This is due to Levy~\cite{MR2741188}, or see~\cite[Theorem~2.53]{MR2884382}.
\end{proof}

\section{A FOD/FOM Bound for $\PP^N$ Endomorphisms}
\label{section:fomfod}

We recall that Theorem~\ref{theorem:FODFOMboundPN} in the introduction
was stated only for number fields~$K$ and their completions, and that
the bound for the FOD/FOM degree of~$f$ then depended only
on~$\dim(\PP^N)$ and~$\deg(f)$. For general fields of
characteristic~$0$, we give a bound for the FOD/FOM degree that
depends also on the period-index exponent\footnote{We recall that the
  Brauer period-index exponent~$\b(K)$ is the smallest positive
  integer such that every $\xi\in\Br(K)$ satisfies
  \text{$\Index(\xi)\mid\Period(\xi)^{\b(K)}$}.  In particular, as
  noted in Remark~\ref{remark:perindsols}, we have $\b(K)=1$ for
  number fields and their completions, so
  Theorem~\ref{theorem:FODFOMboundPN} as stated in the introduction is
  a special case of Theorem~\ref{theorem:FODFOMboundPNB}(b).}  of the
Brauer group of~$K$.

\begin{theorem}
\label{theorem:FODFOMboundPNB}
Let $N\ge1$ and $d\ge2$ be integers, let~$K$ be a field of
characteristic~$0$, and let $f:\PP^N\to\PP^N$ be an endomorphism of
degree~$d$ defined over~$\Kbar$ whose field of moduli is contained
in~$K$.
\begin{parts}
  \Part{(a)}
  There is a field of definition~$L$ for~$f$ satisfying
  \[
    [L:K] \le C_4(N,\#\Acal_f) \cdot \bigl( \#\Acal_f\cdot e^{(N+1)/e} \bigr)^{\b(K)},
  \]
  where as the notation indicates, the constant~$C_4(N,\#\Acal_f)$
  depends only on~$N$ and the order of the automorphism
  group~$\Acal_f$.
  \Part{(b)}
  There is a field of definition~$L$ for~$f$ satisfying
  \[
    [L:K] \le C_5(N,d)^{\b(K)}.
  \]
  \Part{(c)}
  Suppose further that $\Aut(f)=1$.  Then there is a field of
  definition~$L$ for~$f$ satisfying
  \[
    [L:K] \le (N+1)^{\min\{\b(K),N\}}.
  \]
\end{parts}
\end{theorem}

\begin{proof}
(a)\enspace
We assume without loss of generality that~$K$ contains the
group $\bfmu_{N+1}$ of~$(N+1)$'st roots of unity.  We start with the
exact sequence
\[
  \begin{CD}
  1 @>>> \bfmu_{N+1} @>>> \SL_{N+1}(\Kbar) @>>> \PGL_{N+1}(\Kbar) @>>> 1.
  \end{CD}
\]
We define $\hat\Acal_f\subset\SL_{N+1}(\Kbar)$ to be the pull-back
of~$\Acal_f$, and similarly we let
$\hat\Ccal_f\subset\SL_{N+1}(\Kbar)$ be the pull-back of~$\Ccal_f$. We
note that~$\hat\Acal_f$ is an extension of~$\Acal_f$ by~$\bfmu_{N+1}$,
so
\[
  \#\hat\Acal_f = (N+1)\#\Acal_f.
\]
For the remainder of the proof we let
\[
  m = m(\hat\Acal_f) := \text{the exponent of the finite group $\hat\Acal_f$},
\]
so~$m$ is also bounded in terms of~$N$ and~$\#\Acal_f$.  In particular, we may
assume that~$\bfmu_m\subset K$.

Viewing~$\hat\Acal_f$ as a subgroup of~$\GL_{N+1}(\Kbar)$, and using
the fact that the exponent of a group divides its order, we apply
Brauer's theorem (Lemma~\ref{lemma:brauer}) to find a
matrix~$A\in\GL_{N+1}(\Kbar)$ with the property that
\[
  A^{-1}\hat\Acal_f A\subset \SL_{N+1}\bigl(\QQ(\bfmu_m)\bigr) \subset \SL_{N+1}(K).
\]
Using the fact that $\Aut(f^A)=A^{-1}\Acal_fA$, we see that if we
replace~$f$ with~$f^A:=A^{-1}\circ f\circ A$, then
$\Acal_f\subset\PGL(K)$. So we may assume henceforth that
\begin{equation}
  \label{eqn:AutxiinPGLK1}
  \Acal_f\subset\PGL_{N+1}(K)\quad\text{and}\quad \bfmu_{N+1}\subset K.
\end{equation}

We next apply Proposition~\ref{proposition:AfVffacts} to the variety
$V=\PP^N$, which we can do since~$\Acal_f$ is finite and
since~\eqref{eqn:AutxiinPGLK1} tells us in particular that~$\Acal_f$
is defined over~$K$.  We thus get a $1$-cocycle
\[
  \f : G_K \longrightarrow \Acal_f\backslash\Ncal_f
  \quad\text{characterized by}\quad f^\s = f^{\f_\s}.
\]
Thus~$\f$ defines an element of the cohomology set
$H^1(G_K,\Acal_f\backslash\Ncal_f)$. It follows from
Proposition~\ref{proposition:ACG} that we can replace~$K$ with an
extension whose degree is bounded by $C_6(N,\#\Acal_f)\cdot(\#\Acal_f)^{\b(K)}$
so that $\f\in H^1(G_K,\Acal_f\backslash\Ncal_f)$ comes from an element of
$H^1(G_K,\Ccal_f)$.  In other words, there is a 1-cocycle
\[
  \f' : G_K \longrightarrow \Ccal_f
\]
whose image in $H^1(G_K,\Acal_f\backslash\Ncal_f)$ is cohomologous
to~$\f$. This means that there is an element $\g\in\Ncal_f$
with the property that
\[
  \f'_\s \in \Acal_f \g^{-1}\f_\s\g^\s 
  \quad\text{for all $\s\in G_K$.}
\]
(We note that since $\g\in\Ncal_f$, we can multiply both sides by~$\g$
to get $\g\f'_\s \in \Acal_f \f_\s\g^\s$.)

We replace~$f$ with~$f^\g$.  This has the effect of
replacing~$\Acal_f$ by $\Acal_f^\g$, but this is just~$\Acal_f$,
since~$\g\in\Ncal_f$. To determine the $1$-cocycle associated
to~$f^\g$, we compute
\[
  (f^\g)^\s
  = (f^\s)^{\g^\s} 
  = (f^{\f_\s})^{\g^\s} 
  = f^{\f_\s\g^\s} 
  = f^{\g\f'_\s} .
\]
Hence the 1-cocycle associated to~$f^\g$ is the composition
\[
  G_K \xrightarrow{\;\f'\;} \Ccal_f \longrightarrow \Acal_f\backslash\Ncal_f.
\]
By abuse of notation, we write~$f$ instead of~$f^\g$, and we write
$\f:G_K\to\Ccal_f$ for~$\f'$, which is a lift of the 1-cocycle for~$f$
to a 1-cocycle taking values in~$\Ccal_f$.  It remains to find an
appropriate extension of~$K$ over which~$\f$ becomes a coboundary.

Our next task is to pin down more precisely the structure
of~$\Ccal_f$.  We construct a pairing
\begin{equation}
  \label{eqn:pairCfAf}
  \langle\;\cdot\;,\;\cdot\;\rangle : \Ccal_f \times \Acal_f \longrightarrow \Kbar^*
\end{equation}
as follows. Let $\g\in\Ccal_f$ and $\a\in\Acal_f$. Lift~$\g$ and~$\a$ to
elements~$\hat\g\in\hat\Ccal_f$ and~$\hat\a\in\hat\Acal_f$.
Then the fact that~$\a\g=\g\a$ in $\PGL_{N+1}(\Kbar)$ implies that
\[
  \hat\a\hat\g = c(\hat\a,\hat\g) \hat\g\hat\a
  \quad\text{for some scalar $c(\hat\a,\hat\g)\in\Kbar^*$.}
\]
Choosing different lifts of~$\a$ and~$\g$ clearly has no effect on
$c(\hat\a,\hat\g)$, so we  define
\[
  \langle\a,\g\rangle := c(\hat\a,\hat\g) \quad\text{using any choice of lifts.}
\]
It is easy to see from the definition that the
pairing~\eqref{eqn:pairCfAf} is a group homomorphism in each
coordinate, and that it is $G_K$-equivariant.
We define~$\Ccal_f^\circ$ to be the left-kernel, i.e.,
\[
  \Ccal_f^\circ := \bigl\{\g\in\Ccal_f : \text{$\langle\g,\a\rangle=1$
    for all $\a\in\Acal_f$}\bigr\},
\]
and we let $\hat\Ccal_f^\circ$ be the pull-back of~$\Ccal_f^\circ$
to~$\SL_{N+1}(\Kbar)$. By definition we then have
\[
  \hat\g\hat\a=\hat\a\hat\g\quad\text{for all
    $\hat\g\in\hat\Ccal_f^\circ$ and all $\hat\a\in\hat\Acal_f$.}
\]
The pairing induces a homomorphism from~$\Ccal_f$ to the dual
of~$\Acal_f$ with kernel~$\Ccal_f^\circ$, so we obtain a natural
$G_K$-invariant injective homomorphism
\begin{equation}
  \label{eqn:CfCfcirctoAfvee}
  \Ccal_f/\Ccal_f^\circ \longhookrightarrow \Acal_f^\vee := \Hom_{\Kbar}(\Acal_f,\Kbar^*),
   \quad
   \g\longrightarrow\langle\,\cdot\,,\g\rangle.
\end{equation}

We recall that we have a cocycle $\f:G_K\to\Ccal_f$.  We consider the
exact sequence of groups
\[
\begin{CD}
  1
  @>>> \Ccal_f^\circ
  @>>> \Ccal_f
  @>>> \Ccal_f^\circ\backslash\Ccal_f
  @>>> 1,
\end{CD}
\]
leading to an exact sequence of cohomology sets
\[
\begin{CD}
  H^1(G_K,\Ccal_f^\circ)
  @>>> H^1(G_K,\Ccal_f)
  @>>> H^1(G_K,\Ccal_f^\circ\backslash\Ccal_f).
\end{CD}
\]
From~\eqref{eqn:CfCfcirctoAfvee} we obtain the bound
\[
  \#(\Ccal_f^\circ\backslash\Ccal_f) \le  \#\Acal_f^\vee
  \le \#\Acal_f,
\]
so applying Lemma~\ref{lemma:H1fingp},
we can replace~$K$ by a finite extension whose degree is bounded in
terms of~$\#\Acal_f$ so that the 1-cocycle
\[
  G_K \xrightarrow{\;\f\;} \Ccal_f \longrightarrow
  \Ccal_f^\circ\backslash\Ccal_f
\]
is trivial, i.e., so that $\f_\s\in\Ccal_f^\circ$ for all~$\s\in G_K$.
This reduces us to the case that~$\f$ is a $1$-cocycle of the form 
\begin{equation}
  \label{eqn:fGKAAfCf0}
  \f : G_K \longrightarrow \Ccal_f^\circ.
\end{equation}

We next want to use some basic representation theory to
describe~$\Ccal_f^\circ$, but we need to be a bit careful, since the
projective linear group $\PGL_{N+1}(\Kbar)$ does not act
on~$\Kbar^{N+1}$. So instead we use the lifts~$\hat\Acal_f$
and~$\hat\Ccal_f$, which live in~$\SL_{N+1}(\Kbar)$ and thus do act
on~$\Kbar^{N+1}$.  We let $W_1,\ldots,W_r$ be the distinct irreducible
representations of~$\hat\Acal_f$ over the field~$\Kbar$.  Further,
since~$\hat\Acal\subset\SL_{N+1}(K)$, and since we  have already arranged  that~$K$
contains an~$m$'th root of unity, where~$m$ is the exponent of
the group~$\hat\Acal_f$, Brauer's theorem (Lemma~\ref{lemma:brauer}) says
that we may assume that the~$W_i$ are defined over~$K$. (More
precisely, there are~$K$-vector spaces~$W_i'$ on which~$\hat\Acal_f$
act such that~$W_i\cong W_i'\otimes_K\Kbar$ as
$\Kbar[\Acal_f,G_K]$-bimodules.)

We decompose the representation
\[
  \hat\Acal_f \longhookrightarrow \SL_{N+1}(\Kbar)
\]
into a direct sum of irreducible representations, i.e., we choose
a $\Kbar[\hat\Acal_f]$-isomorphism
\begin{equation}
  \label{eqn:psiWieiKN1}
  \psi : \bigoplus_{i=1}^r W_i^{e_i} \xrightarrow[{/\Kbar[\hat\Acal_f]}]{\;\sim\;} \Kbar^{N+1}.
\end{equation}
In this isomorphism, we know that the~$W_i$ are defined over~$K$ and
that the maps in~$\hat\Acal_f$ are defined over~$K$, so Hilbert's
Theorem~90 says that we can find a~$\psi$ that is defined over~$K$,
i.e., so that the map~$\psi$ in~\eqref{eqn:psiWieiKN1} is an
isomorphism of~$\Kbar[\hat\Acal_f,G_K]$-bimodules.\footnote{This is
  standard, so we just sketch the proof. Schur's lemma says that it
  suffices to work with the power~$W^e$ of a single irreducible
  representation.  Let $\t_j:W\hookrightarrow W^e$ be injection on the
  $j$'th factor and $\pi_k:W^e\to W$ projection on the~$k$'th
  factor. Then for every~$\s\in G_K$, the
  map~$\pi_k\psi^{-1}\psi^\s\t_j\in\GL(W)$ commutes with the action
  of~$\hat\Acal_f$, hence Schur's lemma tells us that it is scalar
  multiplcation, say by~$\l_{jk}(\s)$.  Then
  $\s\mapsto\bigl(\l_{jk}(\s)\bigr)_{j,k}$ is a
  $G_K$-to-$\GL_e(\Kbar)$ 1-cocycle, hence by Hilbert's Theorem~90 it
  is the coboundary of some~$M\in\GL_e(\Kbar)$. Using~$M$ to define a
  map $M:W^e\to W^e$ in the obvious way, we find that~$\psi\circ M$ is
  defined over~$K$.}

By definition, the group~$\hat\Ccal_f^\circ$ is the subgroup of
$\SL_{N+1}(\Kbar)$ that commutes with~$\hat\Acal_f$. It is convenient
at this point to extend~$\hat\Ccal_f^\circ$ to include the center
of~$\GL_{N+1}(\Kbar)$, i.e., to include all diagonal matrices, so we
look at~$\Kbar^*\hat\Ccal_f^\circ$. This is the commutator subgroup of~$\hat\Acal_f$
in~$\GL_{N+1}(\Kbar)$, i.e.,
\[
  \Kbar^*\hat\Ccal_f^\circ = \Aut_{\Kbar[\hat\Acal_f]}(\Kbar^{N+1})
  \subset \GL_{N+1}(\Kbar).
\]
Using the $\Kbar[\hat\Acal_f]$-isomorphism~\eqref{eqn:psiWieiKN1}
yields an isomorphism
\begin{equation}
  \label{eqn:autwiei}
  \Aut_{\Kbar[\hat\Acal_f]}\left( \bigoplus_{i=1}^r W_i^{e_i} \right)
  \xrightarrow{\;\sim\;} \Kbar^*\hat\Ccal_f^\circ.
\end{equation}
Applying a general version of Schur's
lemma~\cite[Section~XVII.1]{lang:algebra} to the left-hand side, we
find that
\begin{equation}
  \label{eqn:autwieiglei}
  \Aut_{\Kbar[\hat\Acal_f]}\left( \bigoplus_{i=1}^r W_i^{e_i} \right)
  \cong \prod_{i=1}^r \Aut_{\Kbar[\hat\Acal_f]}\left(  W_i^{e_i} \right)
  \cong \prod_{i=1}^r \GL_{e_i}(\Kbar).
\end{equation}
Alternatively, using the classical version of Schur's
lemma~\cite[Section~2.2]{MR0450380}, the first isomorphism
in~\eqref{eqn:autwieiglei} is a consequence of the fact that for
distinct~$i$ and~$j$, the only~$\hat\Acal_f$-equivariant map
from~$W_i$ to~$W_j$ is the~$0$ map, and the second isomorphism follows
from the fact that for a given~$i$, the only~$\hat\Acal_f$-equivariant
maps from~$W_i$ to~$W_i$ are scalar
multiplications. Combining~\eqref{eqn:autwiei}
and~\eqref{eqn:autwieiglei}, we have identifications
\begin{equation}
  \label{eqn:PNWieiCfGLe}
  \Kbar^{N+1} \xrightarrow{\;\sim\;} \bigoplus_{i=1}^r W_i^{e_i}
  \quad\text{and}\quad
  \Kbar^*\hat\Ccal_f^\circ \xrightarrow{\;\sim\;} \prod_{i=1}^r \GL_{e_i}(\Kbar).
\end{equation}

We recall that we have a cocycle
\[
  \f : G_K \longrightarrow\Ccal_f^\circ.
\]
Using the identifications~\eqref{eqn:PNWieiCfGLe} 
and the fact that the group $\Kbar^*\hat\Ccal_f^\circ$ is the
$\GL_{N+1}(\Kbar)\to\PGL_{N+1}(\Kbar)$ pull-back of~$\Ccal_f^\circ$,
we find that our cocycle has the form
\[
  \f : G_K \longrightarrow \Kbar^* \biggm\backslash \prod_{i=1}^r \GL_{e_i}(\Kbar). 
\]

We next consider the exact sequence
\[
  1 \to \Kbar^* \biggm\backslash \prod_{i=1}^r \Kbar^* \to
  \Kbar^* \biggm\backslash \prod_{i=1}^r \GL_{e_i}(\Kbar)
  \to \prod_{i=1}^r \PGL_{e_i}(\Kbar) \to 1.
\]
We observe that the quotient group on the left is isomorphic to
an $(r-1)$-fold product of copies of~$\Kbar^*$,
and that Hilbert's theorem~90 tells
us that $H^1\bigl(G_K,(\Kbar^*)^{r-1}\bigr)=0$.
Hence taking Galois cohomology yields an injection of pointed sets,
\[
  H^1\bigg(G_K,\Kbar^* \biggm\backslash \prod_{i=1}^r \GL_{e_i}(\Kbar)\biggr) 
  \longhookrightarrow
  \prod_{i=1}^r H^1\bigl(G_K, \PGL_{e_i}(\Kbar)\bigr).
\]
Each of the pointed cohomology sets in the right-hand product admits
an injection into a Brauer group,
\[
  H^1\bigl(G_K,\PGL_e(\Kbar)\bigr)\longhookrightarrow\Br(K)[e],
\]
so we obtain an injection
\[
  H^1\bigg(G_K,\Kbar^* \biggm\backslash \prod_{i=1}^r \GL_{e_i}(\Kbar)\biggr) \\
  \longhookrightarrow
  \prod_{i=1}^r \Br(K)[e_i].
\]
We write the image of our 1-cocycle~$\f$ in the product of Brauer groups as
\[
  \f\longmapsto (\f_1,\ldots,\f_r) \in \prod_{i=1}^r \Br(K)[e_i].
\]
Let~$e_i'$ be the period of~$\f_i$, where~$e_i'$ divides~$e_i$.  By
definition of the Brauer period-index exponent, for each~$i$ we can
find an extension of~$K$ of degree at most~$(e_i')^{\b(K)}$ that
trivializes~$\f_i$, and hence we can find an extension of~$K$ of
degree at most~$(e_1'\cdots e_r')^{\b(K)}$ so that the image of~$\f$
is trivial in~$\prod\Br(K)[e_i]$. We can estimate this degree
using the fact that
\[
  N+1 = \dim\left(\bigoplus_{i=1}^r W_i^{e_i}\right) = \sum_{i=1}^r e_i \dim(W_i) \ge \sum_{i=1}^r e_i,
\]
so the arithmetic-geomtric inequality and elementary calculus yield
\[
  e_1'\cdots e_r' \le e_1\cdots e_r \le \left(\frac{1}{r}\sum_{i=1}^r e_i\right)^r
  \le \left(\frac{N+1}{r}\right)^r \le e^{(N+1)/e}.
\]
(The $e$ in the right-hand expression is the usual~$2.71828\dots$.)
This completes the proof of Theorem~\ref{theorem:FODFOMboundPNB}(a).
\par\noindent(b)\enspace
This follows directly from~(a) and Levy's theorem
(Lemma~\ref{lemma:levy}) which says that~$\Acal_f$ is a finite group whose
order is bounded by a function of~$N$ and~$d$.
\par\noindent(c)\enspace
The assumption that $\Aut(f)=1$ means that we have a
cocycle
$\f:G_K\to\PGL_{N+1}(\Kbar)$ determined by $f^\s=f^{\f_\s}$ whose
triviality in~$H^1(G_K,\PGL_{N+1})$ is equivalent to~$K$ being a FOD
for~$f$. The connecting homomorphism
$\d:H^1(G_K,\PGL_{N+1})\hookrightarrow\Br(K)[N+1]$ sends~$\f$ to an
element~$\d(\f)$ of period dividing~$N+1$. The definition of~$\b$
says that the index of~$\d(\f)$ divides $(N+1)^{\b(K)}$, and the definition
of index says  that there is an extension~$L/K$ of degree
dividing~$(N+1)^{\b(K)}$ such that $\Res_{L/K}\d(\f)=1$ in $\Br(L)$. It follows
that~$\Res_{L/K}(\f)=1$ in $H^1(G_L,\PGL_{N+1})$, and hence that~$L$
is a FOD for~$f$. This proves half of~(c).

For the other half, we use the theory of Severi--Brauer varieties,
i.e., varieties~$X$ that are defined over~$K$ and admit a
$\Kbar$-isomorphism to~$\PP^N$. We refer the reader
to~\cite{MR3309942} or~\cite[X~\S6]{MR554237} for the basic facts that
we use.  The cocycle~$\f:G_K\to\PGL_{N+1}(\Kbar)$ is associated to a
Severi--Brauer variety~$X_\f$.  We claim that there is a field $L/K$
satisfying
\[
  X_\f(L)\ne\emptyset\quad\text{and}\quad [L:K]\le (N+1)^N.
\]
From this it will follow that~$X_\f\times_KL$ is a trivial
Severi-Brauer variety~\cite[X~\S6]{MR554237}, i.e.,~$X_\f$
is~$L$-isomorphic to~$\PP^N$, and hence that the cocycle~$\f$
trivializes over~$L$.  To prove our claim, we note that since~$X_\f$
is defined over~$K$ and is~$\Kbar$-isomorphic to~$\PP^N$, the
anti-canonical bundle~$\Kcal_{X_\f}^{-1}$ on~$X_\f$ is defined
over~$K$ and is very ample. The associated linear system has dimension
equal to
$\dim{H^0}\bigl(\PP^N,\Ocal_{\PP^N}(N+1)\bigr)=\binom{2N+1}{N}$, so we
obtain an embedding
\begin{equation}
  \label{eqn:XembedKX1}
  \iota : X_\f \longrightarrow |\Kcal_{X_\f}^{-1}| \cong \PP^{\binom{2N+1}{N}-1}_K
\end{equation}
that is defined over~$K$. The degree of the
embedding~\eqref{eqn:XembedKX1}, i.e., the number of geometric points
in the intersection of~$\iota(X_\f)$ with a generic linear subspace of
complementary dimension, is~$(N+1)^N$;
cf.\ \cite[Exercise~I.7.1(a)]{hartshorne}. Intersecting~$\iota(X_\f)$
with a linear subspace defined over~$K$ gives points on~$\iota(X_\f)$
defined over a field of degree~$L$ with $[L:K]\le(N+1)^N$. 
\end{proof}

\section{An Alternative Approach using Quotient Varieties}
\label{section:FOMFODviaquotients}

The material in this section may be useful in an alternative approach
to FOD/FOM problems for endomorphisms $f:V\to V$ in which one tries to
rigidify the map~$f$ by specifying the position of marked points,
e.g., (pre)periodic points. One way to do this is to look at the map
that~$f$ induces on the quotient variety~$V\GITQuot\Acal_f$, and
twist~$V\GITQuot\Acal_f$ to obtain a map defined over the FOM of~$f$,
as in the following result.

\begin{proposition}
\label{proposition:AfVffactsquotient}
We continue with the notation from
Section~$\ref{section:FOMFODforVariety}$ and the assumptions in
Proposition~$\ref{proposition:AfVffacts}$.
\begin{parts}
\Part{(a)}
The quotient variety
\[
  \bar V_f := V\GITQuot\Acal_f
\]
is defined over~$K$, and~$f$ descends to give a
$\Kbar$-morphism\footnote{To be notationally consistant, we should use
  the horrible notation~$\bar f_f$ for this map, but instead will
  simply use~$\bar f$.}
\[
\bar f:\bar V_f\to\bar V_f.
\]
\Part{(b)}
Composing the $1$-cocycle~$\f$ with the map
$\Acal_f\backslash\Ncal_f\to\Aut(\bar V_f)$ gives a $1$-cocycle
\[
  \hat\f : G_K \xrightarrow{\;\;\f\;\;} \Acal_f\backslash\Ncal_f
  \longhookrightarrow \Aut(\bar V_f).
\]
Let~$\bar V_f^\f$ be the~$\Kbar/K$-twist of~$\bar V_f$ determined
by~$\hat\f$, and let~$F$ be a $\Kbar$-isomorphism
\[
  F : \bar V_f^\f\xrightarrow{\;\sim/\Kbar\;} \bar V_f\quad\text{satisfying}\quad
  \hat\f_\s\circ F^\s = F.
\]
Then the map
\[
  \bar f^F  : \bar V_f^\f \longrightarrow \bar V_f^\f
\]
is defined over~$K$, where as usual $\bar f^F$ is our notation for
$F^{-1}\circ \bar f\circ F$.
\Part{(c)}  
Let $P\in V(\Kbar)$ be a point such  that $F^{-1}(\bar P)\in V_f^\f(K)$.
Then for all  $\s\in G_K$ we have
\[
  \f_\s^{-1}(P) = \Acal_f P^\s,
\]
where this notation indicates that
since~$\f_\s\in\Acal_f\backslash\Ncal_f$, the function~$\f_\s^{-1}$ sends a
point in~$V(\Kbar)$ to the~$\Acal_f$-orbit of a point.
\end{parts}
\end{proposition}
\begin{proof}
(a)\enspace  
We are given~\eqref{eqn:Affinite} that~$\Acal_f$ is finite, and in the
category of algebraic varieties, quotients by finite groups of
automorphisms always exist. Then the assumption~\eqref{eqn:AfdefK}
that~$\Acal_f$ is defined over~$K$ implies that the
quotient variety is defined over~$K$. 
\par\noindent(b)\enspace  
We compute
\[
  (\bar f^F)^\s =  (F^\s)^{-1} \bar f^\s F^\s
  = (\hat\f_\s^{-1}F)^{-1} (\hat\f_\s^{-1}\bar f\hat\f_\s) (\hat\f_\s^{-1}F)
  = \bar f^F.
\]
Hence~$\bar f^F$ is defined over~$K$.
\par\noindent(c)\enspace
We compute
\begin{align*}
  F^{-1}(\bar P)
  &= F^{-1}(\bar P)^\s
  &&\text{since~$F^{-1}(\bar P)$ is defined over~$K$,} \\*
  &= (F^{-1})^\s(\bar P^\s) \\*
  &= F^{-1}\circ\hat\f_\s(\bar P^\s)
  &&\text{since $\hat\f_\s\circ F^\s=F$.}
\end{align*}
Applying~$\hat\f_\s^{-1}\circ F$ to both sides, we find that
\[
  \hat\f_\s^{-1}(\bar P) = \bar P^\s .
\]
Lifting this to~$V$, it says precisely that~$\f_\s^{-1}(P)$ is
the~$\Acal_f$-orbit of~$P^\s$.
\end{proof}

The next result says that we can find large numbers of periodic points
that avoid any specified proper closed subvariety, where in general
for a morphism $f:V\to V$, we use the standard notation,
\[
  \Per_n(f) := \bigl\{ P\in V(\Kbar) : f^n(P)=P \bigr\}.
\]

\begin{proposition}
\label{proposition:numbperptsonZ} 
Let $d\ge2$, and let $Z\subsetneq\PP^N$ be a proper closed subvariety
of~$\PP^N$.  Then for every $r\ge1$ there exists an $n=n(N,d,r,Z)$
such that
\[
  \#\bigl(\Per_n(f) \setminus Z\bigr) \ge r \quad\text{for all  $f\in\End_d^N$.}
\]
\end{proposition}
\begin{proof} 
We set the notation
\[
\Per^*_{n,t}(f) := \Per_n(f) \setminus \bigcup_{i=1}^{t} \Per_i(f),
\]
i.e., $\Per_{n,t}^*(f)$ is the set of periodic points of~$f$  whose
exact period divides~$n$ and is at least equal to~$t+1$.

For $n\ge1$ and $t\ge1$, define a (possibly reducible) subvariety
$Y_{n,t}\subseteq\End_d^N$ by the condition
\[
  Y_{n,t} := \bigl\{ f\in\End_d^N : \Per^*_{n,t}(f) \subseteq Z \bigr\}.
\]
We note that~$Y_{n,t}$ is a subvariety, since the map $f\to\Per_n(f)$ is a
morphism from~$\End_d^N$ to an appropriate Chow variety, and the
condition that $\Per^*_{n,t}(f)\subseteq Z$ leads, via elimination
theory, to an algebraic condition on the coefficients of $f$. We let
\[
  X_{n,t} := \bigcap_{k=1}^n Y_{k,t}.
\]
Equivalently, the set~$X_{n,t}$ is characterized by
\[
X_{n,t} = \left\{f\in\End_d^N:
\begin{tabular}{@{}l@{}}
  every periodic point of $f$ of exact period\\
  between $t+1$ and $n$ lies on the subvariety $Z$\\
  \end{tabular}\right\}.
\]
We observe that
\[
  X_{1,t}\supseteq X_{2,t}\supseteq X_{3,t}\supseteq \cdots\,.
\]
A decreasing sequence of varieties must stabilize, and hence there is
an $m=m(N,d,t,Z)$ having the property that
\[
  X_{m+i,t} = X_{m,t}\quad\text{for all $i\ge0$.}
\]

We claim that $X_{m,t}=\emptyset$. Suppose not. Then we can find a map
\[
  f\in X_{m,t}=\bigcap_{k=1}^\infty X_{k,t}.
\]
It would follow that all but finitely many periodic points of~$f$ lie
on~$Z$, i.e., every $f$-periodic point of period strictly larger
than~$t$ would lie on~$Z$. However, by assumption,~$Z$ is a proper closed
subvariety of~$\PP^N$, so this contradicts a theorem of
Fakhruddin~\cite[Corollary~5.3]{MR1995861} stating that the periodic
points of~$f$ are Zariski dense in~$\PP^N$.

We now know that for every $t\ge1$ there is an $m=m(N,d,t,Z)$ so that
$X_{m,t}=\emptyset$. Hence  every $f\in\End_d^N$ has a periodic point~$P_f$
whose exact period satisfies
\[
  t < \text{Period}(P_f) \le m(t),
\]
where to ease notation, we write $m(t)$ for $m(N,d,t,Z)$,
since~$N$,~$d$, and~$Z$ are fixed.

We apply this last statement recursively. Thus we start with $t=1$, so
for every~$f$ can find a point~$P_{f,1}\notin Z$ whose exact period is
less than~$m(1)$. We then apply the statement with $t=m(1)$, which
gives us a point~$P_{f,2}\notin $ satisfying
\[
  m(1) < \text{Period}(P_{f,2}) \le m^{\circ2}(1),
\]
where as usual, $m^{\circ2}(1)$ means $m(m(1))$.  We observe that $P_{f,2}\ne
P_{f,1}$, since $\text{Period}(P_{f,1}) \le m(1)$ and
$\text{Period}(P_{f,2}) > m(1)$. Repeating the process with $t=m^{\circ2}(1)$
yields a third periodic point $P_{f,3}\notin Z$ with exact period
between $m^{\circ2}(1)+1$ and $m^{\circ3}(1)$, hence distinct from~$P_{f,1}$
and~$P_{f,2}$. Proceeding in this fashion, we see that for every~$f\in\End_d^N$
we can find distinct periodic points~$P_{f,1},\ldots,P_{f,r}$ for~$f$
that do not lie on~$Z$ and with periods at most~$m^{\circ r}(1)$. We observe that~$m^{\circ r}(1)$
depends only on~$N$,~$d$,~$r$ and~$Z$.
Hence taking
\[
n:=\operatorname{LCM}\bigl(1,2,\ldots,m^{\circ r}(1)\bigr)
\]
completes the proof of Proposition~\ref{proposition:numbperptsonZ}.
\end{proof}


\begin{acknowledgement}
The authors would like to thank Michael Rosen for his helpful advice.
\end{acknowledgement}




\begin{thebibliography}{10}

\bibitem{MR2181874}
Gabriel Cardona and Jordi Quer.
\newblock Field of moduli and field of definition for curves of genus 2.
\newblock In {\em Computational aspects of algebraic curves}, volume~13 of {\em
  Lecture Notes Ser. Comput.}, pages 71--83. World Sci. Publ., Hackensack, NJ,
  2005.

\bibitem{MR2060023}
A.~J. de~Jong.
\newblock The period-index problem for the {B}rauer group of an algebraic
  surface.
\newblock {\em Duke Math. J.}, 123(1):71--94, 2004.

\bibitem{MR1443489}
Pierre D{\`e}bes and Jean-Claude Douai.
\newblock Algebraic covers: field of moduli versus field of definition.
\newblock {\em Ann. Sci. \'Ecole Norm. Sup. (4)}, 30(3):303--338, 1997.

\bibitem{moduliportrait2017}
John~R. Doyle and Joseph~H. Silverman.
\newblock Moduli spaces for dynamcial systems with portraits, 2018.
\newblock in preparation.

\bibitem{MR1995861}
Najmuddin Fakhruddin.
\newblock Questions on self maps of algebraic varieties.
\newblock {\em J. Ramanujan Math. Soc.}, 18(2):109--122, 2003.

\bibitem{MR1233388}
Benson Farb and R.~Keith Dennis.
\newblock {\em Noncommutative algebra}, volume 144 of {\em Graduate Texts in
  Mathematics}.
\newblock Springer-Verlag, New York, 1993.

\bibitem{hartshorne}
Robin Hartshorne.
\newblock {\em Algebraic {G}eometry}, volume~52 of {\em Graduate Texts in
  Mathematics}.
\newblock Springer-Verlag, New York, 1977.

\bibitem{MR3230378}
Rub\'en~A. Hidalgo.
\newblock A simple remark on the field of moduli of rational maps.
\newblock {\em Q. J. Math.}, 65(2):627--635, 2014.

\bibitem{MR3309942}
Benjamin Hutz and Michelle Manes.
\newblock The field of definition for dynamical systems on {$\mathbb P^N$}.
\newblock {\em Bull. Inst. Math. Acad. Sin. (N.S.)}, 9(4):585--601, 2014.

\bibitem{lang:algebra}
Serge Lang.
\newblock {\em Algebra}, volume 211 of {\em Graduate Texts in Mathematics}.
\newblock Springer-Verlag, New York, third edition, 2002.

\bibitem{MR2741188}
Alon Levy.
\newblock The space of morphisms on projective space.
\newblock {\em Acta Arith.}, 146(1):13--31, 2011.

\bibitem{MR3030517}
Andrea Marinatto.
\newblock The field of definition of point sets in {$\mathbb{P}^1$}.
\newblock {\em J. Algebra}, 381:176--199, 2013.

\bibitem{MR0094360}
T.~Matsusaka.
\newblock Polarized varieties, fields of moduli and generalized {K}ummer
  varieties of polarized abelian varieties.
\newblock {\em Amer. J. Math.}, 80:45--82, 1958.

\bibitem{mortonsilverman:rationalperiodicpoints}
Patrick Morton and Joseph~H. Silverman.
\newblock Rational periodic points of rational functions.
\newblock {\em Internat. Math. Res. Notices}, (2):97--110, 1994.

\bibitem{MR2567424}
Clayton Petsche, Lucien Szpiro, and Michael Tepper.
\newblock Isotriviality is equivalent to potential good reduction for
  endomorphisms of {$\mathbb{P}^N$} over function fields.
\newblock {\em J. Algebra}, 322(9):3345--3365, 2009.

\bibitem{GSM186}
Bjorn Poonen.
\newblock {\em Rational points on varieties}, volume 186 of {\em Graduate
  Studies in Mathematics}.
\newblock American Mathematical Society, Providence, RI, 2017.

\bibitem{MR0450380}
Jean-Pierre Serre.
\newblock {\em Linear representations of finite groups}.
\newblock Springer-Verlag, New York-Heidelberg, 1977.
\newblock Translated from the second French edition by Leonard L. Scott,
  Graduate Texts in Mathematics, Vol. 42.

\bibitem{MR554237}
Jean-Pierre Serre.
\newblock {\em Local {F}ields}, volume~67 of {\em Graduate Texts in
  Mathematics}.
\newblock Springer-Verlag, New York, 1979.
\newblock Translated from the French by Marvin Jay Greenberg.

\bibitem{shimura:fldofdef}
Goro Shimura.
\newblock On the field of definition for a field of automorphic functions. {I},
  {I}{I}, {I}{I}{I}.
\newblock {\em Ann. of Math. (2)}, 80, 81, 83:160--189, 124--165, 377--385,
  1964, 1965, 1966.

\bibitem{MR0125113}
Goro Shimura and Yutaka Taniyama.
\newblock {\em Complex multiplication of abelian varieties and its applications
  to number theory}, volume~6 of {\em Publications of the Mathematical Society
  of Japan}.
\newblock The Mathematical Society of Japan, Tokyo, 1961.

\bibitem{silverman:fieldofdef}
Joseph~H. Silverman.
\newblock The field of definition for dynamical systems on {$\mathbb{P}\sp 1$}.
\newblock {\em Compositio Math.}, 98(3):269--304, 1995.

\bibitem{silverman:modulirationalmaps}
Joseph~H. Silverman.
\newblock The space of rational maps on {$\mathbb{P}\sp 1$}.
\newblock {\em Duke Math. J.}, 94(1):41--77, 1998.

\bibitem{MR2316407}
Joseph~H. Silverman.
\newblock {\em The {A}rithmetic of {D}ynamical {S}ystems}, volume 241 of {\em
  Graduate Texts in Mathematics}.
\newblock Springer, New York, 2007.

\bibitem{MR2884382}
Joseph~H. Silverman.
\newblock {\em Moduli {S}paces and {A}rithmetic {D}ynamics}, volume~30 of {\em
  CRM Monograph Series}.
\newblock American Mathematical Society, Providence, RI, 2012.

\end{thebibliography}



\end{document}